\documentclass[12pt]{amsart}

\usepackage{color}
\usepackage{amssymb}
\usepackage{graphicx}

\usepackage{psfrag}

\def\e{\varepsilon}

\def\Z{\mathbb{Z}}

\def\G{\Gamma}
\def\FP{{\rm FP}}

\def\<{\langle}
\def\>{\rangle}
\def\-{\overline}

\def\onto{\to \hskip -0.35cm \to}

\def\wFP{{\rm{wFP}}}

\def\eenv{\exists\!\! {\rm{E}}}

\newtheorem{theorem}{Theorem}[section]
\newtheorem{lemma}[theorem]{Lemma} 
\newtheorem{proposition}[theorem]{Proposition}

\newtheorem{corollary}[theorem]{Corollary}

\newtheorem{definition}[theorem]{Definition}
\newtheorem{addenda}[theorem]{Addenda}

\theoremstyle{definition}  

\newtheorem{example}[theorem]{Example}

\newtheorem{remark}[theorem]{Remark}

\numberwithin{equation}{section}

\catcode`\@=11
\def\serieslogo@{\relax}
\def\@setcopyright{\relax}
\catcode`\@=12

\usepackage[table]{xcolor}
\definecolor{lightgray}{gray}{0.9}  

\begin{document}

\title[Binary Subgroups of Direct Products]{Binary Subgroups of Direct Products}

\author[Bridson]{Martin R.~Bridson}
\address{Martin R.~Bridson\\
Mathematical Institute \\
Andrew Wiles Building\\
Oxford OX2 6GG \\ 
UK} 
\email{bridson@maths.ox.ac.uk}



\begin{abstract}  We explore an elementary construction that produces finitely presented groups
with diverse homological finiteness properties -- the {\em binary subgroups},
$B(\Sigma,\mu)<G_1\times\dots\times G_m$. 
These full subdirect products require strikingly few generators.
If each $G_i$  is finitely presented, $B(\Sigma,\mu)$  is finitely presented. 
When the $G_i$ are non-abelian limit groups (e.g. free or surface groups), the $B(\Sigma,\mu)$
provide new examples of finitely presented, residually-free groups that
do not have finite classifying spaces and are not of Stallings-Bieri type. These examples
settle a question of Minasyan relating different notions of rank for residually-free groups. Using binary subgroups,
we prove that if $G_1,\dots,G_m$ are perfect groups, each requiring at most $r$ generators,  then
$G_1\times\dots\times G_m$  requires at most $r \lfloor \log_2 m+1 \rfloor$ generators.
\end{abstract} 

\maketitle

\centerline{\em For my friend Vaughan Jones, in memoriam}

\section{Introduction} In the study of infinite groups, the problem of deciding which 
finite subsets of a given group $G$ generate finitely presentable subgroups is notoriously difficult:  
there is no algorithm that can solve this problem when $G$ is a direct product of two non-abelian
free groups for example \cite{miller}. 
Even when one knows that finite presentations exist, constructing them can remain 
impossibly difficult \cite{BW}. 

In this article we will explore an elementary construction that provides a new source of finitely presented 
subdirect products of groups; we call them {\em binary subgroups}. This construction takes as input two positive integers $m$ and $r$,
an $m$-tuple of $r$-generator groups with markings $\mu=(\mu_i:F_r\onto G_i)_{i=1}^m$,
and an $r$-tuple of finite sets of positive integers $\Sigma=(\sigma_1,\dots,\sigma_r)$, with each $|\sigma_i|=m$.
It outputs a subgroup $B(\Sigma,\mu)<G_1\times\dots\times G_m$ that is
finitely presented if all of the $G_i$ are finitely presented. $B(\Sigma,\mu)$ is constructed from the binary expansions of the elements of 
$\bigcup\sigma_i$. I hope to convince the reader that this construction is an intriguing source of groups that merits further
study.

An important feature of the groups $B(\Sigma,\mu)$ is that they
require strikingly few generators: see Corollaries \ref{c:hall} and \ref{c:ashot}. The case where each $G_i$ is free will be of particular
interest. Indeed the construction of $B(\Sigma,\mu)$ is motivated in large part by my
work with  Howie, Miller and Short \cite{BHMS1, BHMS2}
exploring the structure of subgroups of direct products of free groups, surface groups and limit groups.
The VSP Theorem from \cite{BHMS2} will play a prominent role in our discussion, as will homological finiteness properties.  
\medskip 
 
In addition to economical generation and finiteness properties, 
we shall be concerned with the  co-nilpotency class of  binary subgroups.
{\em Co-nilpotency} is a characteristic of 
finitely presented, residually free groups that originates in \cite{BM}; it played a prominent role in \cite{BHMS1, BHMS2}. Let $S<\Lambda_1\times\dots\times \Lambda_m$ be a finitely presented subgroup of a direct product of non-abelian limit groups (e.g. free groups), and suppose that we have reduced to the case where $S$ intersects each $\Lambda_i$ non-trivially and the coordinate projections $S\to \Lambda_i$ are onto (in other
words, $S$ is a {\em full subdirect product}). It is proved in
\cite{BHMS2} that in this case $S$ contains some term of the lower central series of a subgroup of finite index 
$D_0<\Lambda_1\times\dots\times \Lambda_m$. If the first such term is $\gamma_{c+1}(D_0)$, then the co-nilpotency class of $S$ is defined to be $c$.  
Thus $S$ has co-nilpotency class $0$ if it is of finite index and it has co-nilpotency class $1$ if it is virtually the kernel of a map from 
a product of non-abelian limit groups to an infinite abelian group. Groups of the latter kind are, by definition, of Stallings-Bieri type; they have
been extensively studied in connection with higher finiteness properties \cite{stall, bieri}. One of the achievements of \cite{BHMS1} was to exhibit
the first examples of finitely presented, residually free groups whose co-nilpotency class is greater than $1$. The binary subgroups that
we shall describe provide a more elementary construction of such groups with a variety of co-nilpotency classes. For example,
the co-nilpotency classes of the groups $B_0(m)$ and $B_1(m)$ constructed below go to infinity with $m$.

\subsection{Paradigms: the Binary Subgroups $B_0(m)$ and $B_1(m)$}
Before delving into more technical matters, I want to illustrate the basic
construction of this article with explicit examples. 
We work with a free group of rank $r$, denoted by $F$, and fix a basis $\{a_1,\dots,a_r\}$.
Throughout, $F^m$ will denote the $m$-th direct power $F\times\dots\times F$.

We consider two subgroups 
$B_0(m)< B_1(m) < F^m $. The subgroup $B_0(m)$ is generated by the $r\lfloor 1 + \log_2 m\rfloor$ elements $a_{ij}$ defined as follows:
consider the array $A(m)$ with $m$ columns and $\lfloor 1 + \log_2 m\rfloor$ rows where column $k$ is the binary expansion of $k$, with units
at the top; for $j=0,\dots,\lfloor \log_2 m\rfloor$ let $\varepsilon_j(m)$ be the word in the alphabet $\{0,1\}$
that is the $j$-th row of $A(m)$ (so $k=\sum_j\varepsilon_j(k)2^j$); we treat $\varepsilon_j(m)$ as a 
multi-index and, for each $a_i$, define $a_{ij}\in F^m$ to be the element obtained by raising $a_i$ to this multi-index, with the convention 
$a_i^1=a_i$ and $a_i^0=1$ (the identity). 

For example, when $m=18$,  
\begin{table}[ht]
\caption{The binary array $A(m)$ for $m=18$}
\begin{center} 
\begin{tabular}{r|rrrrrrrrrrrrrrrrrrr}
  \hline
& 1 & 2 & 3 & 4 & 5 & 6 & 7 & 8 & 9 & 10 & 11 & 12 & 13 & 14 & 15 & 16 & 17 & 18 \\
  \hline
  $\varepsilon_0$   & 1 & 0 & 1 & 0 & 1 & 0 & 1 & 0 & 1 & 0 & 1 & 0 & 1 & 0 & 1 & 0 & 1 & 0 \\
  $\varepsilon_1$  & 0 & 1 & 1 & 0 & 0 & 1 & 1 & 0 & 0 & 1 & 1 & 0 & 0 & 1 & 1 & 0 & 0 & 1\\
  $\varepsilon_2$   & 0 & 0 & 0 & 1 & 1 & 1 & 1 & 0 & 0 & 0 & 0 & 1 & 1 & 1 & 1 & 0 & 0 & 0\\
  $\varepsilon_3$  & 0 & 0 & 0 & 0 & 0 & 0 & 0 & 1 & 1 & 1 & 1 & 1 & 1 & 1 & 1 & 0 & 0 & 0\\ 
  $\varepsilon_4$  & 0 & 0 & 0 & 0 & 0 & 0 & 0 & 0 & 0 & 0 & 0 & 0 & 0 & 0 & 0 & 1 & 1 & 1\\ 
   \hline
\end{tabular}
\end{center}
\end{table}   
\begin{table}[ht]
\caption{The binary elements $a_{ij}$ when $m=18$}
\begin{center} 
\begin{tabular}{rrrrrrrrrrrrrrrrrrrr}
$a_{i0}=$ &($a_i$,&1,&$a_i$,&1,&$a_i$,&1,&$a_i$,&1,&$a_i$,&1,&$a_i$,&1,&$a_i$,&1,&$a_i$,&1,&$a_i$,&1)\\
$a_{i1}=$ &(1,&$a_i$,&$a_i$,&1,&1,&$a_i$,&$a_i$,&1,&1,&$a_i$,&$a_i$,&1,&1,&$a_i$,&$a_i$,&1,&1,&$a_i$)\\
$a_{i2}=$ &(1,&1,&1,&$a_i$,&$a_i$,&$a_i$,&$a_i$,&1,&1,&1,&1,&$a_i$,&$a_i$,&$a_i$,&$a_i$,&1,&1,&1)\\
$a_{i3}=$ &(1,&1,&1,&1,&1,&1,&1,&$a_i$,&$a_i$,&$a_i$,&$a_i$,&$a_i$,&$a_i$,&$a_i$,&$a_i$,&1,&1,&1)\\
$a_{i4}=$ &(1,&1,&1,&1,&1,&1,&1,&1,&1,&1,&1,&1,&1,&1,&1,&$a_i$,&$a_i$,&$a_i$)\\
\end{tabular}
\end{center}
\end{table}  
 
The subgroup $B_1(m)<F^m$ is obtained from $B_0(m)$ by adjoining the {\em diagonal elements} $\delta_i:=(a_i,\dots,a_i)$,
$$B_1(m)= \langle B_0, \delta_1,\dots,\delta_r\rangle<F^m.$$
In our discussion of finiteness properties, the most natural ones to consider are finite presentation, type ${\rm{FP}}_k$
and {\em weak ${\rm{FP}}_k$} (written ${\rm{wFP}}_k$). A group
$G$ is of type ${\rm{FP}}_k$ if the trivial $\Z G$-module $\Z$ has a projective resolution in which the modules
up to dimension $k$ are finitely generated, and $G$ has type ${\rm{wFP}}_k$ if $H_i(G_0,\Z)$ is finitely generated for
 all $i\le k$ and all subgroups $G_0<G$ of finite index.  Finite presentation implies ${\rm{FP}}_2$ and
 ${\rm{FP}}_k$ implies ${\rm{wFP}}_k$.
 
We use the standard notation for the lower central series of a group,
$\gamma_1(G)=G$ and $\gamma_{c+1}(G)= [G, \gamma_c(G)]$.  We also use the standard term {\em rank}
and write $d(G)$ for the cardinality of a smallest generating set for $G$.

\begin{theorem} \label{t:main} For $r\ge 2, m\ge 3$,
\begin{enumerate}
\item the rank of {{$B_0(m)<F^m$}} is {{$r(\lfloor 1+\log_2 m\rfloor)$}};
\item  the rank of {{$B_1(m)<F^m$}} is {{$r(\lfloor 2+\log_2 m\rfloor)$}}; 
\item  {{$B_0(m)$}} contains {{$\gamma_{m-1}(F^m)$}};
\item  {{$B_1(m)$}} contains $\gamma_{c}(F^m)$, where
{{$c =  \lceil (m-1)/2\rceil$}};
\item  {{$B_0(m)$}} is {{finitely presented}} but {{not}} of type {{${\rm{wFP}}_3$}};
\item  if $m\le 4$ then {{$B_1(m)=F^m$}};
\item if $m\ge 5$ then {{$B_1(m)$}} is {{finitely presented}}, type ${\rm{wFP}}_3$ but {{not ${\rm{wFP}}_4$}}; 
\item  {{$\forall c\ \exists$ polynomial $p_c(t)$}} of degree $c$ such that if $m> p_c(\log_2 m)$ and $D_0<F^m$
is of finite index, then $ {{\gamma_c(D_0)\not\subseteq B_1(m)}}$.  
\end{enumerate}
\end{theorem}  

One can replace $F$ by a non-abelian limit group in this theorem -- see Theorem \ref{t:limit}.

\bigskip

\def\rkk{{\rm{rk}_F}}

The general binary subgroup $B(\Sigma)<F^m$ is obtained from $\Sigma=(\sigma_1,\dots,\sigma_r)$ by writing the binary expansions
of the elements of each $\sigma_i$ as the columns of an array, as above -- the $i$-th array $A_i$ will have $l_i$ rows where $l_i = \max\{\lfloor 1+ \log_2 x\rfloor \mid x\in\sigma_i\}$; the generators $a_{ij}\in B(\Sigma)$ are obtained by raising $a_i$ to the multi-index given by the $j$-th row of $A_i$ (see Section \ref{s:general} for more details).

Given any $m$-tuple of $r$-generator groups with markings
$(\mu_i:F\onto G_i)_{i=1}^m$, we take the image of  $B(\Sigma)<F^m$ under the
epimorphism $\mu = (\mu_1,\dots,\mu_m)$ to define $B(\Sigma,\mu)$. With this notation we have:

\begin{theorem}\label{t:general1} $B(\Sigma,\mu)<G_1\times\dots\times G_m$  contains the $(m-1)$st term
of the lower central series and is closed in the profinite topology. Moreover, there is an algorithm that, 
given $\Sigma,\mu$ and finite presentations of $G_1,\dots,G_m$, will construct a finite presentation for  $B(\Sigma,\mu)$. 
If the $G_i$ are limit groups, then there is an algorithm that will determine for each $k\ge 2$ whether 
$B(\Sigma,\mu)$ is of type $\wFP_k$.
\end{theorem}

The passage from $B(\Sigma,\mu)$ to  $B(\Sigma)<F^m$ is an example of how one can push forward subdirect products:
see Sections \ref{s:push-and-pull} and \ref{s:5.2}, where we also examine the results of pulling back subgroups.

\subsection{Economical generation}
Philip Hall \cite{hall} initiated the study of economical generating sets for direct powers of finite perfect groups. The
theory was subsequently developed extensively by Jim Wiegold \cite{wiegold} and others, and elements of it were extended to 
cover infinite groups \cite{WW}: if $G$ is finitely generated and perfect, then $G^m$ requires at most $O(\log m)$
generators.
Theorems \ref{t:main} and \ref{t:general1} allow us to extend this statement about $d(G^m)$ to direct products of distinct groups. 
For if $G_1,\dots, G_m$ each  require at most $r$ generators,  then by pushing forward $B_0(m) <F^m$
we obtain a subgroup $B< G_1\times\dots\times G_m$ with $d(B)\le r \lfloor 1 + \log_2 m\rfloor$ that contains
a term of the lower central series of the product, and if the $G_i$ are perfect then each term of the lower central
series is equal to the entire product.
 
\begin{corollary}\label{c:hall}
If $G_1,\dots,G_m$ are perfect and $d(G_i) \le r$, then
$ d(G_1\times\dots\times G_m) \le r \lfloor 1 + \log_2 m\rfloor .$
\end{corollary}

For direct powers of finitely presented superperfect groups, the logarithmic growth of $d(G^m)$ can be promoted to a 
polylogarithmic bound on the number of relators needed to present $G^m$ (see \cite{B:concise}), but we do not have
such a result for more general direct products.

\subsection{Economical embeddings of residually free groups} In  \cite{BHMS2} an algorithm is established  that,
given a finite presentation of a residually free group $G$, will construct a canonical embedding of $G$ into its
{\em existential envelope}  $\eenv (G)$,  which is a direct product of finitely many limit groups (see Section \ref{s:limit} below for 
definitions). If $G$ is
centerless, then
the number of direct factors in $\eenv (G)$ is the product rank of $G$, i.e. the largest integer $\rkk(G)$ such that $G$ contains 
a direct power of non-abelian free groups $F_2^{\rkk(G)}$.
If a direct product of finitely many limit groups contains $G$ then it contains $\eenv (G)$ and has at least $\rkk(G)$ direct factors.

In \cite{ashot}, Ashot Minasyan proved that if $d(G)=k$ then $\eenv (G)$  is a product of at most $\exp(C_\e k^{2+\e})$
limit groups, where $\e>0$ is arbitrary and $C_\e$ is a constant that goes to infinity with $\e$.
He asked if this upper
bound could be improved to a polynomial in $k$. Theorem \ref{t:main} answers his question in the negative: one cannot do better than an exponential function of $k$. To see this, note that parts (3) and (4) of the theorem imply that
$B_0(m)$ and $B_1(m)$ both have product rank $m$.

\begin{corollary}\label{c:ashot}
There exist sequences of finitely presented, centerless, residually-free groups $B(m)$ with 
product rank $\rkk(B(m))=m$ requiring only $d(B(m)) = O(\log m)$ generators.
\end{corollary}

This article is organised as follows. In Section \ref{s:background} we gather some basic facts about 
nilpotent groups and subdirect products, and state the VSP Theorem.
In Section \ref{s:limit} we recall some elements of the theory of residually free groups and present the results
that we need 
from \cite{BHMS1, BHMS2, Koch, Kuckuck} relating finiteness properties of subdirect products to coordinate projections.
Section \ref{s:proof} contains the proof of Theorem \ref{t:main}. In Section \ref{s:general} we discuss
the general construction of binary subgroups, the process of pushing them forward and
pulling them back, and certain redundancies in the general construction.
 In the final section we discuss possible variations, open questions, and directions for future work, including an
 enticing connection with coding theory.
\bigskip
  
\noindent{\bf Acknowedgements:} I thank Ben Green and Yonathan Fruchter for 
helpful discussions concerning the connection to coding theory described in Section 
\ref{s:last}.
I thank Claudio Llosa Isenrich for encouraging comments and for drawing my attention to Ashot Minasyan's question  \cite{ashot}. And  
I thank my co-authors from \cite{BHMS1, BHMS2}, Jim Howie, Chuck Miller and Hamish Short, for many years of stimulating 
and enjoyable conversations about subdirect products of limit groups.

\section{Preliminaries}\label{s:background}

\subsection{Nilpotent groups} With the exception of Proposition \ref{p:poly}, the material in this section is standard
and can be found in Chapter 11 of the classical reference  \cite{marshall-hall}.
The terms of the lower central series of a group $G$ are defined
by $\gamma_1(G)=G$ and $\gamma_{n+1}(G) = [G, \gamma_n(G)]$. If
$\gamma_{c+1}(G)=1$ but $\gamma_c(G)\neq 1$ then $G$ is said to be {\em nilpotent of class $c$}.  
The {\em free nilpotent group} of class $c$ and rank $k$ is $N(k,c):=F_k/\gamma_{c+1}(F_k)$; every 
$k$-generator nilpotent group of class at most $c$ is a quotient of this group.

If $\{a_1,\dots,a_k\}$ is a basis for $F_k$, then the basic commutators
$[a_{i_1},\dots,a_{i_n}]$ of weight $n$  project under the natural map $F_k\to F_k/\gamma_{n+1}(F_k)$
to a basis of the free abelian group $\gamma_n(F_k)/\gamma_{n+1}(F_k)$, the rank of which  is given by the Witt formula
$$
W_n(k) = \frac{1}{n} \sum_{d|n} \mu(d) k^{n/d}
$$
where $\mu(d)$ is the M\"{o}bius function.  
By making repeated use of this observation, we see that $N(k,c)$ is a polycyclic group whose 
Hirsch length (cohomological dimension) is $h(k,c) = \sum_{i=1}^c W_i(k)$.  
A simple induction then shows that every subgroup of $N(k,c)$ requires at most $h(k,c) $ generators.  
The crude approximation  $W_n(k) \le k^n$ gives $h(k,c)\le \sum_{i=1}^c k^i$, which is all that we shall need for the
polynomial in the last item of Theorem \ref{t:main}.

\begin{lemma} \label{l2}
If $H$ is a subgroup of a $k$-generator nilpotent group $G$ of class at most $c$, then $d(H)\le h(k,c)$. 
\end{lemma}

\begin{proof}
By choosing an epimorphism  $N(k,c)\onto G$  and replacing $H$ with its preimage in $N(k,c)$ we may assume that
 $G=N(k,c)$, and this case is covered by the preceding discussion. 
\end{proof}

The following is the main output that we require from this subsection.

\begin{proposition} \label{p:poly}
Let $\Lambda_1,\dots,\Lambda_m$ be finitely generated groups each of which  
maps onto a non-abelian free group. Suppose $D=\Lambda_1\times\dots\times\Lambda_m$ has a $k$-generator
subgroup $S$ that contains $\gamma_c(D_0)$ where $D_0<D$ is a subgroup of finite
index. Then $m\le  h(k,c)/W_c(2)$.
\end{proposition}

\begin{proof}  
The hypothesis yields a surjection $D\onto F_2^m$ which maps  $\gamma_i(D)$ onto $\gamma_i(F_2)^m$.
We compose $D\onto F_2^m$ 
with the canonical surjection  $F_2^m\onto  N(2,c)^m$. 
Let $\overline{S}$ and $\overline{D}_0$ be the images of $S$ and $D_0$ under this composition. Then $d(\overline{S})\le k$ and
$\overline{S}$ is a nilpotent group of class at most $c$ containing $\gamma_c(\overline{D}_0)$. The latter has finite index 
in  the $m$-th direct power of $\gamma_{c}(F_2)/\gamma_{c+1}(F_2)$,
which is a free abelian group of rank $mW_c(2)$, hence $d(\gamma_c(\overline{D}_0)) = mW_c(2)$.  Therefore
$
m W_c(2) \le h(k,c),
$
by Lemma \ref{l2}.
\end{proof}

The following well-known facts are proved by induction on the nilpotency class.

\begin{lemma}\label{l:ab-enough}
Let $N$ be a finitely generated nilpotent group. A subset $\Theta\subset N$ generates $N$ if and only if
the image of $\Theta$ in the abelianization $N^{\rm{ab}}=H_1(N,\Z)$ generates $N^{\rm{ab}}$.
\end{lemma}

\begin{lemma}\label{l:tor-finite} If a nilpotent group $N$ is generated by a finite set of elements of finite order, then
$N$ is finite.
\end{lemma}

\subsection{Subdirect Products and the VSP Theorem} 
A subgroup of a direct product $S< G_1\times\dots\times G_m$ is 
termed a {\em{subdirect product}} if $S$ projects onto each $G_i$; it is  
a {\em{full}} subdirect product if, in addition, $S\cap G_i\neq 1$ for all $i$. If one wishes to understand the finitely generated
subgroups of direct products of groups drawn from a class that is closed under the taking of finitely generated subgroups,
then one can reduce to the study of full subdirect products by projecting away from factors where $S\cap G_i= 1$
and replacing each $G_j$ by the projection of $S$ to $G_j$; free groups form one such class and limit groups form another.

We refer the reader to \cite{BM} for a more thorough account of subdirect products.

The results of the present paper rely heavily on the following theorem from \cite{BHMS2}.

\begin{theorem}[The VSP Theorem \cite{BHMS2}] \label{t:VSP}
Let $S<D=G_1\times\dots\times G_m$ be a
subgroup of a direct product of finitely presented groups. If the projection $p_{ij}(S)<G_i\times G_j$ has
finite index for all $1\le i<j\le m$, then 
\begin{enumerate}
\item $S$ is finitely presented;
\item $\gamma_{m-1}(D_0)<S$ for  some $D_0<D$ of finite index;
\item $S$ is closed in the profinite topology on $D$.
\end{enumerate}
Moreover, there is an algorithm that, 
given finite presentations of the groups $G_i$ and a finite set
$\Theta\subset D$ generating $S$, will construct a finite presentation $\langle\Theta\mid R\rangle$ for $S$. 
\end{theorem} 

Item (1) of this theorem is considerably deeper than item (2), the latter being an instance of  the following general
observation \cite{BM}.
 
\begin{lemma}\label{l:easy-proj} Let $k\ge 2$ and let $G_1,\dots,G_m$ be groups.
If $S<D=G_1\times\dots\times G_m$ projects onto each $k$-tuple of factors,
then $\gamma_c(D)<S$ where $c=\lceil (m-1)/(k-1)\rceil$.
\end{lemma}

To prove this, one first notes that since $S$ projects onto each factor, $S\cap G_i$ is normal in $G_i$, so it is enough to show that
$S\cap G_i$ contains all iterated commutators $[x_1,\dots,x_c]$. And such commutators can be obtained explicitly from the hypothesis;
for example, when  $m=3$ and $k=2$,  one has
$
([x,y], 1, 1) = [\, (x,\ast,1),\ (y,1,\ast)\, ].
$ 

\begin{corollary}\label{c:H_1enough} Let $k\ge 2$.
If $G_1,\dots,G_m$ are finitely generated and
$S<D=G_1\times\dots\times G_m$ projects onto each pair of factors,
 then the projection of $S$ to a $k$-tuple of factors $T=G_{i_1}\times\dots\times G_{i_k}$
will be surjective if and only if the composition $S\to T\to T^{\rm{ab}}=H_1(T,\Z)$ is surjective. 
\end{corollary}

\begin{proof}
$S$ contains $\gamma_{m-1}(D)$ and hence the image of $S$ in $T$ contains $\gamma_{m-1}(T)$, so
$S\to T$ is surjective if and only if $S\to T/\gamma_{m-1}(T)$ is surjective. The result now follows from Lemma \ref{l:ab-enough}.
\end{proof}

\begin{corollary}\label{c:useless} Let  $G_1,\dots,G_m$ be finitely generated groups, each of which is generated by elements of
finite order.  If $S<D=G_1\times\dots\times G_m$ projects onto each pair of factors,
then $S$ has finite index in $D$.
\end{corollary}

\begin{proof}
$S$ contains $\gamma_{m-1}(D)$ and Lemma \ref{l:tor-finite} tells us that  $D/\gamma_{m-1}(D)$ is finite.
\end{proof}

\begin{remark}[{\tt{Algorithms}}]\label{r:algo}
Corollary \ref{c:H_1enough} makes the process of determining whether a subgroup $S$ of a direct product of finitely presented groups maps onto
all $k$-tuples of factors algorithmically decidable, provided one has the {\em a priori} knowledge that $S$ projects to each pair of factors.
Similarly, the problem of deciding if $S$ projects to a subgroup of finite index in each $k$-tuple of factors is decidable  given this
knowledge.
\end{remark}

\subsection{Pushing forward and pulling back}\label{s:push-and-pull}
Given an $m$-tuple of group epimorphisms  $\mu_i :H_i\onto G_i$, the {\em push-forward} of a subgroup
$S<H_1\times\dots\times H_m$ is its image $\mu(S)<G_1\times\dots\times G_m$
under $\mu = (\mu_1,\dots,\mu_m)$.
The {\em pull-back} of a subgroup $T<G_1\times\dots\times G_m$ is $\mu^{-1}(T)$.  
 
The following lemma is immediate from the definitions.

\begin{lemma} \label{l:push-and-pull} If $S<G_1\times\dots\times G_m$ projects to $k$-tuples of factors, then
so does any pull-back or push-forward of $S$. And if $S$ contains the $c$-th term of the lower central
series of the product, then so does any pull-back or push-forward of $S$.
\end{lemma}

\section{Residually free groups, limits groups and coordinate projections} \label{s:limit}
A group $G$ is {\em residually free} if
for every $g\in G\smallsetminus\{1\}$ there exists a homomorphism $\phi:G\to F$, with $F$ free,
such that $\phi(g)\neq 1$; and $G$ is {\em{fully residually free}} if for every finite set $S\subset G$
there is a homomorphism $G\to F$ that is injective on $S$.  
Finitely generated, fully residually free groups are more commonly called {\em{limit groups}}, reflecting the powerful geometric
approach to their study initiated by Sela \cite{sela} {\em{et seq.}}
This important and
much-studied class of groups contains free groups (obviously), free abelian groups, and the fundamental groups of all
closed surfaces of euler characteristic at most $-2$.

The study of finitely presented, residually free groups $\G$ reduces to the study of subdirect products 
of limit groups because there is a canonical (and algorithmic) embedding 
 $\G\hookrightarrow \eenv(\G)$ where $\eenv(\G)$, the {\em existential envelope} of $\G$, is a direct product of limit groups 
 \cite{BMR}, \cite{BHMS2}.
(At most one of the direct factors is abelian, and the intersection of $\G$ with this is the centre of $\G$.)
The following theorem is therefore an important tool in classifying which residually free groups are finitely presented.
This is the main result of \cite{BHMS1}. The first item
provides a converse to the VSP Theorem in the setting of limit groups, while the second item
points to the key role that finiteness properties play in understanding residually free groups.

\begin{theorem}[\cite{BHMS2}]
Let $S<D=\Lambda_1\times\dots\times\Lambda_m$ be a full subdirect product, where the $\Lambda_i$
are non-abelian limit groups. 
\begin{enumerate}
\item  If $S$ is finitely presented, then the projection $p_{ij}(S)<\Lambda_i\times\Lambda_j$ has finite index for $1\le i < j\le m$.
\item If  $S$ is not of finite index in $D$, then $S$  is not of type $\wFP_m$.
\end{enumerate}
\end{theorem} 

Point (2) was improved upon by Kochloukova \cite{Koch} and Kuckuck \cite{Kuckuck}, who tied the finiteness properties of $S$
more closely to projections to $k$-tuples. The following statement combines their results.

\begin{theorem}[\cite{Koch},  \cite{Kuckuck}] \label{t:KK}
For $2\le k\le m$, a full subdirect product of non-abelian limit groups
$S< \Lambda_1\times\dots\times \Lambda_m$ projects to a subgroup of finite index in each $k$-tuple of factors
$\Lambda_{i(1)}\times\dots\times \Lambda_{i(k)}$  if
and only if $S$ is of type ${\rm{wFP}}_k$.  
\end{theorem} 

Henceforth we shall abbreviate the phrase ``projects to a subgroup of finite index in"  to ``virtually surjects to".
 
\section{Proof of Theorem \ref{t:main}} \label{s:proof}

With the results of the previous sections in hand, most of the assertions of Theorem \ref{t:main} will follow once we
verify that $B_0(m)<F^m$ surjects to pairs of factors but does not virtually surject certain triples, while
 $B_1(m)<F^m$ surjects all triples of factors but does not virtually surject certain  $4$-tuples.
 
 \medskip
 \noindent{\em Notation:} We write $L_i\cong F$ for the $i$-th summand of $F^m$ throughout this section. 

\subsection{Proof that $B_0(m)<F^m$ projects onto pairs of factors}\label{s:not-triples} 

Consider columns $p$ and $q$ of the array $A(m)$, with $p<q$. Since $p<q$, there is at least one bit in the binary expansion
of $p$ where it has $0$ while $q$ has $1$; if this place is bit $k$, then in row $k$ we have $(\e_k(p),\e_k(q))=(0,1)$,
hence $a_{ik}$ projects to $(1,a_i)\in L_p\times L_q$. 
Since $p\neq 0$, there is also a non-zero bit in its expansion; if this
is place $l$, then $(\e_l(p),\e_l(q))\in\{(1,0),\, (1,1)\}$, hence   $a_{il}$ maps to either $(a_i,1)$ or $(a_i,a_i)$ in $L_p\times L_q$.
In either case, we see that $\{(a_i,1),\,(1,a_i),\, (a_i,a_i)\}$ is contained in the projection of $B_0(m)$.
This holds for $i=1,\dots,r$ and all $1\le p< q\le m$, so $B_0(m)$ projects onto each pair of factors, as claimed.\qed

\begin{remark}\label{remark}
 Having worked through the above proof with the original notation, we shall abbreviate several of
 the similar proofs that follow by noting that the essential
feature of $A(m)$ was that the subgroup $C<\mathbb{Z}^m$ generated by the row vectors 
$\e_i$ projects onto the summands $\Z^2<\Z^m$ corresponding to
each pair of coordinates. 
\end{remark}

In the spirit of the above remark, the key fact for the following lemma is that the projection of $C$ to $\mathbb{Z}^3\times 1<\Z^m$ only
has dimension $2$.

\begin{lemma}
$B_0(m)<F^m$ does not virtually surject to certain triples of factors.  
\end{lemma}

\begin{proof}
The non-zero entries in columns $1,2,3$ are concentrated in two rows, namely $\e_0$ and $\e_1$, so the projection of $B_0(m)$
to $H_1(L_1\times L_2\times L_3, \Z)\cong\Z^{3r}$ has rank at most (in fact exactly) $2r$.
\end{proof}

\subsection{Proof that $B_1(m)<F^m$ projects to triples}

Consider columns $p<q<s$ of the array $A(m)$. The essential point to prove this time is that we get a basis for 
$\mathbb{Z}^3$ by adjoining the element  $(1,1,1)$ (coming from $\delta_i=(a_i,\dots,a_i)$)
to the elements $(\e_j(p), \e_j(q), \e_j(s))_j$ obtained by reading the entries in columns $p,q,s$ in each
row of the array.

Since $p<s$, there is at least one place in the binary expansion
of $p$ where it has $0$ while $s$ has $1$; if this place is bit $k$, then from columns $p,q,s$ of row $k$ we have either $(0,1,1)$
or $(0,0,1)$. Similarly, as $p<q$, in some row we get either $(0,1,1)$ or $(0,1,0)$. If we get two distinct vectors 
from this argument, then together with $(1,1,1)$ (coming from $\delta_i$)
they will span $\mathbb{Z}^3$. This fails if we get $(0,1,1)$ twice, in which case
we use the condition $q<s$ to get a vector $(1,0,1)$ or $(0,0,1)$: adding either to $\{(1,1,1),\, (0,1,1)\}$ gives a basis for $\mathbb{Z}^3$. \qed

\begin{lemma}\label{l:not4} If $m\ge 5$ then
$B_1(m)<F^m$ does not virtually surject to certain 4-tuples of factors.  
\end{lemma}

\begin{proof} We focus on columns $2$ to $5$ of the array $A(m)$
and argue that the projection of $B_1(m)<F^m$ to  $Z:=H_1(L_2\times L_3\times L_4\times L_5,\Z)\cong\Z^{4r}$ only
has rank $3r$. 

To see this, consider the rank-$4$ summand $Z_1<Z$ that is spanned by the images of 
$(a_1,1,1,1),$ $(1,a_1,1,1),\ (1,1,a_1,1),\ (1,1,1,a_1).$
The only rows of  $A(m)$ that have non-zero entries in columns $2,3,4,5$ are $\e_0, \e_1, \e_2$, where we have
$$ (0,1,0,1),\ (1,1,0,0),\ (0,0,1,1).$$
Thus the projection of $\< a_{1j} \mid j=0,1,2,\dots\>$ to $Z_1$ has rank $3$.
The diagonal element $\delta_1$ contributes $(1,1,1,1)$, but we already have this as the sum of the images of $a_{11}$ and $a_{12}$. A similar
argument applies to the summand $Z_i<Z$ corresponding to the generator $a_i\in F$. Thus the image of 
$B_1(m)$ in $Z$ has rank $3r$, as claimed. 
\end{proof} 

\subsection{The remainder of the proof} In the light of the preceding results in this section,
items (3) and (4) of Theorem \ref{t:main}
are covered by Lemma \ref{l:easy-proj}, while items (5) and (7) are
covered by Theorems \ref{t:VSP} and  \ref{t:KK}. We shall prove the remaining items in order.

Let $k= \lfloor 1+\log_2 m\rfloor$. We defined $B_0(m)$ by taking a generating set $a_{ij}$ with $rk$ elements, so
certainly $d(B_0(m))\le rk$. To see that $d(B_0(m))= rk$, we observe that $B_0(m)$ maps onto a direct product of $k$ copies of $F=F_r$:
the desired map is the coordinate projection onto $\Lambda^{[2]}:=L_1\times L_2 \times L_4\dots\times L_{2^{k-1}}$; in this projection,
$a_{ij}$ maps to $a_i\in L_{2^j}$. This proves assertion (1) of the theorem.

The proof of item (2) is similar. In this case, to see that  {{$B_1(m)$}} requires at least {{$r(k+1)$}} generators
we prove that  its projection to the abelianization of $\Lambda^{[2]}\times L_3$ is surjective. The new
difficulty that we face is that in the rank-$3$ summand of
$V_i<H_1(L_1\times L_2\times L_3, \Z)$ corresponding to the basis element $a_i\in F$, 
the only contributors among the $a_{ij}$ are $a_{i0}$
and $a_{i1}$, who contribute $(1,0,1,0,...)$ and $(0,1,1,0,...)$ respectively. 
But the element $\delta_i\prod_{j>1}a_{ij}^{-1}$ contributes $(1,1,1,0,...)$ to $\Lambda^{[2]}\times L_3$, ensuring that the
image of $B_1(m)$ contains $V_i$. And  $a_{ij}$ maps to $a_i\in L_{2^j}$ for $j>1$, as before. 
Thus  $B_1(m)$ maps onto
$H_1(\Lambda^{[2]}\times L_3, \Z)\cong \Z^{r(k+1)}$, so
 $d(B_1(m)) \ge r(k+1)$.

The proof of item (6) is similar to that of item (2). The three rows of $A(4)$ are $(1,0,1,0),$ $ (0,1,1,0),\, (0,0,0,1)$ and together
with $\delta=(1,1,1,1)$ these span $\Z^4$. In the spirit of remark \ref{remark}, this translates easily into the fact that 
$B_1(4)$ contains each of the $4r$ generators of $F^4$. Alternatively, we can observe (i) that  $B_1(4)$ maps onto
$H_1(F^4,\Z)$ because it contains each of the $r$ summands $\Z^4$ corresponding to a choice of basis vector, and (ii)
$B_1(4)$ contains $[F^4,F^4]$,  by Lemma \ref{l:easy-proj}.

Item (8) is an immediate consequence of (2) and Proposition \ref{p:poly}. In more detail, 
since $h(r,c)\le\sum_{i=1}^cr^i$ and $\lfloor 2 + \log_2 m\rfloor > \log_2 m$, it suffices
to define $p_c(t) = \frac{1}{W_c(2)}\sum_{i=1}^ct^i$.
\qed  

\subsection{Binary subgroups of direct products of limit groups}

Let  $\Lambda_1,\dots,\Lambda_m$ be  non-abelian limit groups, and suppose that there exist epimorphisms $\mu_i:F_r\to \Lambda_i$
that induce an isomorphism on abelianizations. Of particular interest is the case
where each $\Lambda_i$ is the fundamental group of a closed surface of
genus $g$, with $r=2g$ and $\mu_i:F_r\to \Lambda_i$ a standard choice of generators.  
Let {{$B_j(m; \mu)<D=\Lambda_1,\dots,\Lambda_m$}} be the image of $B_j(m)$ under $\mu = (\mu_1,\dots,\mu_m)$.

\begin{theorem} \label{t:limit} With the above notation, for all non-abelian limit groups $\Lambda_i$ such that 
$\mu_i:F_r\to \Lambda_i$ induces an isomorphism on $H_1(-,\Z)$:
\begin{enumerate}
\item the rank of {{$B_0(m;\mu)$}} is {{$r(\lfloor 1+\log_2 m\rfloor)$}};
\item  the rank of {{$B_1(m;\mu)$}} is {{$r(\lfloor 2+\log_2 m\rfloor)$}}; 
\item  {{$B_0(m;\mu)$}} contains {{$\gamma_{m-1}(D)$}};
\item  {{$B_1(m;\mu)$}} contains $\gamma_{c}(D)$, where
{{$c =  \lfloor 1+(m-1)/2\rfloor$}};
\item  {{$B_0(m;\mu)$}} is {{finitely presented}} but {{not}} of type {{${\rm{wFP}}_3$}};
\item  if $m\le 4$ then {{$B_1(m;\mu)=D$}};
\item if $m\ge 5$ then {{$B_1(m;\mu)$}} is {{finitely presented}}, type ${\rm{wFP}}_3$ but {{not ${\rm{wFP}}_4$}}; 
\item  {{$\forall c\ \exists$ polynomial $p_c(t)$}} such that if $m> p_c(\log_2 m)$ and $D_0<\Lambda_1\times\dots\times\Lambda_m$
is of finite index, then $ {{\gamma_c(D_0)\not\subseteq B_1(m)}}$.   
\end{enumerate}
\end{theorem}

The polynomial $p_c(t)$ depends only on $c$, not the particular $\Lambda_i$ considered: again we may take
$p_c(t) = \frac{1}{W_c(2)}\sum_{i=1}^ct^i$.

\begin{proof}
The proofs of Theorem \ref{t:main} and the supporting results were crafted so that they extend {\em{mutatis mutandis}} to this
setting.
\end{proof} 

In fact, our earlier proofs apply to more general settings: the interested reader will have little
difficulty in verifying the following observations. (Lemma \ref{l:push-and-pull} is
needed here.)

\begin{addenda} \begin{enumerate}
\item[(i)] If one drops the requirement that $\mu$ induces an isomorphism on $H_1(-,\Z)$, then items
(3), (4), (6) and (8) of Theorem \ref{t:limit} remain valid while the equalities in (1) and (2) are replaced by inequalities
$d(B_j(m;\mu))\le r(\lfloor j+1+\log_2 m\rfloor)$. The positive 
assertions in (5) and (7) remain valid but the negative assertions may fail.

\item[(ii)]
If one weakens the hypotheses further, requiring only that each $\Lambda_i$ is an $r$-generator
 finitely presented group that maps onto a non-abelian free group, then the same set of conclusions as in {\rm{(i)}} remains valid. 
 
 \item[(iii)]
If one weakens the hypotheses yet again, requiring only that each $\Lambda_i$ is an $r$-generator finitely presented group, then  the only additional
item to become invalid is (8).
\end{enumerate}
\end{addenda}

\section{Binary Subgroups: the general construction}\label{s:general}

We maintain the notation that $F$ is the free group of rank $r$ with basis $\{a_1,\dots,a_r\}$.

First we consider arbitrary binary subgroups of direct powers of free groups.

\begin{definition} \label{d:d1}
Let $m\ge 2$ and $r\ge 2$ be integers. 
For $i=1,\dots,r$, let $\sigma_i=(x_{i1},\dots,x_{im})$ be a list of distinct positive integers. 
Let $l_i = \max_k\ \lfloor 1+ \log_2 x_{ik}$. 
We associate to $\sigma_i$ the $l_i$-by-$m$ matrix $A[\sigma_i]$ whose $k$-th column is
the binary expansion of $x_{ik}$. More
precisely, the $(j,k)$-entry of $A[\sigma_i]$ is $\e_j(k)$ if $x_{ik}=\sum_j\varepsilon_j(k)2^j$. 
We treat the $j$-th row of $A[\sigma_i]$ as a multi-index and  define
$a_{ij}\in F^m$ by raising the generator $a_i\in F$ to this multi-index (cf.~Table 2).

Let  $\Sigma = (\sigma_1,\dots,\sigma_r)$. We define
$B(\Sigma)<F^m$ to be the subgroup generated by $\{a_{ij} \mid 1\le i\le r,\ 1\le j\le l_i\}$.
\end{definition}

\begin{example}
$B_0(m)$ is the group one gets by taking $\sigma_1=\dots=\sigma_r$ and $x_{ik}=k$.
\smallskip

$B_1(m)$ is the group one gets by taking $\sigma_1=\dots=\sigma_r$ and $x_{ik}=2k+1$.  
\end{example}

The proof in Section 4.1 that $B_0(m)$ surjects pairs of factors applies equally well to $B(\Sigma)$.

\begin{lemma}\label{l:more}
The projection of $B(\Sigma)<F^m$ to each pair of factors is surjective.
\end{lemma}

From the VSP Theorem we deduce:

\begin{proposition}
$B(\Sigma)<F^m$ is finitely presented and contains $\gamma_{m-1}(F^m)$.
\end{proposition}

\begin{remark}[{\tt{Algorithms}}]\label{r:explore}
Following Remark \ref{r:algo}, with Lemma \ref{l:more} in hand we see that there is
an algorithm that allows one to explore, for each $k$, whether $B(\Sigma)<F^m$ (virtually)
projects onto all $k$-tuples of factors and hence to bound
the higher finiteness properties and co-nilpotency class of $B(\Sigma)$ using Theorem \ref{t:KK} and Lemma \ref{l:easy-proj}.
 \end{remark}

\subsection{Pushing forward and the proof of Theorem \ref{t:general1}}

The following defines the most general binary subgroups that we consider.

\begin{definition} \label{d:d2}
Let  $\Sigma = (\sigma_1,\dots,\sigma_r)$ be as in definition \ref{d:d1} and let 
$(\mu_i:F\onto G_i)_{i=1}^m$ be an   $m$-tuple of $r$-generator groups. We define $B(\Sigma,\mu)$
to be the image in $G_1\times\dots\times G_m$ of  $B(\Sigma)<F^m$ under the
epimorphism $\mu = (\mu_1,\dots,\mu_m)$.
\end{definition}

We invoke the VSP Theorem once more -- in tandem with Lemma \ref{l:push-and-pull} and Remark \ref{r:algo} it yields:

\begin{theorem} \label{t:general} For all $m$-tuples of finitely presented $r$-generator 
groups $(\mu_i:F\onto G_i)_{i=1}^m$ and all $r$-tuples $\Sigma$ of lists of $m$ distinct positive integers,
the subgroup
$B(\Sigma,\mu)<G_1\times\dots\times G_m=:D$ is finitely presented, closed in the profinite topology, and contains $\gamma_{m-1}(D)$.
Moreover, there is an algorithm that, given finite presentations of $G_1,\dots,G_m$, will construct a finite presentation for 
$B(\Sigma,\mu)$. And for each $k\ge 2$ there is an algorithm that will determine if $B(\Sigma,\mu)$ (virtually) surjects all $k$-tuples
of factors $G_{i_1}\times\dots\times G_{i_k}$.
\end{theorem}

This completes the proof of Theorem \ref{t:general1}, because
the additional assertion concerning finiteness properties of limit groups is covered by Theorem \ref{t:KK}.

\subsection{Pulling back}\label{s:5.2}
  Let $B(\Sigma,\mu)<G_1\times\dots\times G_m$ be as above. 
Given any $m$-tuple of epimorphisms $(\pi_i : H_i\to G_i)_{i=1}^m$, with the groups $H_i$ finitely presented,
 we consider the pull-back
$\pi^{-1}B(\Sigma,\mu)<E:=H_1\times\dots\times H_m$.
In the light of Lemma \ref{l:push-and-pull} and the VSP Theorem we see that $\pi^{-1}B(\Sigma,\mu)$
will again be finitely presented, closed in the profinite topology, and contain $\gamma_{m-1}(E)$.
And there is an algorithm that, given finite presentations of $H_1,\dots,H_m$, will construct a finite presentation for $\pi^{-1}B(\Sigma,\mu)$.

This construction is of particular interest in the case where the $H_i$ are limit groups and the $G_i$ are free groups.
We shall explore it further in \cite{BL}.  

\subsection{Redundancy, variations and standard forms}\label{ss:variations}

It is natural to wonder to what extent $B(\Sigma)<F^m$ determines $\Sigma$. Deciding isomorphism among 
the groups $B(\Sigma)$ does not seem to be a trivial matter. It is clear that $\Sigma$ is not uniquely determined:
permuting the order of the sets $\sigma_i$ corresponds to permuting the chosen basis of $F$, and in the case
$\sigma=\sigma_1=\dots=\sigma_m$ permuting the order of the elements in $\sigma$ merely permutes the
direct factors of $F$. (One might remove these ambiguities by insisting on a standard form in which the $\sigma_i$
and their elements are listed in a natural order.) Beyond this, even in the case $\sigma=\sigma_1=\dots=\sigma_r$ there are 
further redundancies:
if no column of the matrix $A[\sigma]$ contains an entry $1$ in both of the rows $j_1$ and $j_2$, then one can replace
row $j_1$ by the sum of these rows
 to obtain $A[\sigma']$ with $\sigma'\neq \sigma$. This operation leaves  $B(\Sigma)$ invariant, merely changing the preferred
 generating set $a_{ij}$ by performing the Nielsen move $a_{ij_1}\mapsto a_{ij_1}a_{ij_2}$ for each $i$. 

An important point to note is that when we perform such row operations, we are constrained by wanting to work within 
the context of integer matrices with entries in $\{0,1\}$. If we were to work instead over the field with
2 elements, then we would be free to perform row operations; we could then carry out Gaussian elimination to
transform the matrix $A[\sigma]$ into 
a canonical form without altering $B(\Sigma)$. This freedom can be gained at the expense of switching from the free group
$F=F_r$ to the  free product $W_r$ of $r$ copies of $\Z/2\Z$. One might be tempted to think
that this switch would lead to a neater theory, and it would certainly tighten
the connection with binary linear codes, but the following observation shows that this alternative is useless if one is
interested in constructing groups with  exotic finiteness properties.
 
We fix an epimorphism $F\to W_r$ and extend it coordinate-wise to obtain $\mu:F^m\to W_r^m$. 

\begin{proposition}\label{p:sad} For all choices of $\Sigma$,  the subgroup  $B( \Sigma,\mu)<W_r^m$  is of finite index.
\end{proposition}

\begin{proof} Follows immediately from Lemma \ref{l:more} and Corollary \ref{c:useless}.
\end{proof}
 
\section{Connections and challenges} \label{s:last}
 
$B_0(m)$ has the minimum number of generators   
among the binary subgroups of $F^m$. It is striking that one only needs to add $r$ additional generators to obtain $B_1(m)$,
which surjects all triples of factors and therefore has stronger finiteness properties. 
It is natural to ask
what the minimum number of generators is for a binary subgroup that surjects $4$-tuples of factors, $5$-tuples, {\em{etc}}.
More generally, it might be interesting to explore the nature of small collections of integers $\Sigma$ such that 
$B(\Sigma)<F^m$ maps onto $k$-tuples of factors, where $k<m$ is fixed. In particular, one might count
such collections and  explore the asymptotics of  the counting functions.

Note that Corollary \ref{c:H_1enough} reduces these issues to problems in linear algebra over $\Z$; in 
the case $\Sigma=(\sigma,\dots,\sigma)$, one would just be exploring submatrices of $A[\sigma]$, and the problem is so natural
that it has surely been studied in other contexts.
Note too that Theorem \ref{t:KK} provides a direct connection with the 
homological finiteness properties of $B(\Sigma)$.
The challenge of understanding how the co-nilpotency class of $B(\Sigma)$ varies with $\Sigma$ seems harder.

 \subsection{Binary Codes}
It seems likely that some answers to the type of question raised above
might be found by exploring the link to linear codes that I shall now sketch. I am grateful to Ben Green for suggesting this connection, and I am grateful to Yonathan Fruchter for further conversations about it.
  
A {\em binary linear code} of length $m$ is a linear subspace $C$ of the vector space $\mathbb{F}_2^m$; the {\em rank}
of the code is the dimension of $C$. The dual code $C^{\perp}<\mathbb{F}_2^m$ consists of the
vectors orthogonal to
$C$ with respect to the usual inner product $x.y = \sum x_iy_i$.
The non-zero vectors in $C$ are called codewords, and the minimum Hamming distance of a 
codeword from the zero vector is the {\em weight} of the code.  A linear code of length $m$, dimension $k$ and weight $d$
 is called an $[m,k,d]$-code. Such codes have been intensively studied in the context of error correction in computer science 
 and communication.
 
Each of the arrays that we used to define binary subgroups $B<F^m$ defines a binary code of length $m$:
 we take  $C$ to be the span of the rows of the array, read as elements of $\mathbb{F}_2^m$. Conversely, given
 an $[m,k,d]$-code with  basis $c_1,\dots,c_k$, we define the  $k\times m$ array (matrix) whose rows are the vectors $c_i$
 to be $A[\sigma]$, thereby defining $\sigma$ and hence $B[C]:=B(\Sigma)<F_r^m$, where $\Sigma = (\sigma,\dots,\sigma)$. 
 (More generally, given an $r$-tuple of codes, we could take a basis for each code to obtain matrices $A[\sigma_i]$ defining 
 $(\sigma_1,\dots,\sigma_r)$ -- but for the purposes of this discussion the first case is enough.) 
 
 Let $C_0(m)<\mathbb{F}_2^m$ be the code defined by the array for $B_0(m)<F^m$  (see Table 1). 
 The argument that we used in Section \ref{s:not-triples} to prove that $B_0(m)<F^m$ does not surject the triple of factors
 $L_1\times L_2\times L_3 <F^m$ shows that $C_0(m)$ does not project onto $\mathbb{F}_2^3\times 0< \mathbb{F}_2^m$.
 Indeed it shows that $(1,1,1,0,\dots)$ is orthogonal to $C_0(m)$, hence the dual code $C_0(m)^{\perp}$ has weight at most $3$. In
 fact, it has weight exactly $3$, because if there were a vector $y\in C_0(m)^{\perp}$ whose only non-zero coordinates were
 $i$ and $j$, then the projection of  $B_0(m)<F^m$ to $L_i\times L_j$ (even $H_1(L_i\times L_j, \mathbb{F}_2)$) would not be surjective.
Likewise, the dual of the linear code $C_1(m)<\mathbb{F}_2^m$ corresponding to $B_1(m)<F^m$ has weight $4$.
 (In retrospect, one sees that $B_0(m)$ and $B_1(m)$ correspond to the Hadamard and extended Hadamard codes, which
are dual to the Hamming and extended Hamming codes, whose weights are well-known to be $3$ and $4$.) 
 In general, if a code $C$ has length $m$ and the weight of the dual code $C^{\perp}$ is $d'>2$
 then the corresponding binary subgroup $B[C]<F^m$ fails to project onto some $d'$-tuple of factors in $F^m$
 and hence is not of type $\wFP_{d'}$. 
  
 These observations suggest that in order to produce binary subgroups $B<F^m$ with a small number of generators 
 and strong finiteness properties, one should turn to the literature on coding theory 
 in search of $[m,k,d]$ binary codes with $k$ small and the weight $d'$ of the dual code large, keeping in mind
 that we are obliged to work over $\Z$ rather than $\mathbb{F}_2$ (cf.~Proposition \ref{p:sad}).
 Yonathan Fruchter has made progress in this direction (private communication) and points out the relevance of selective families and
 superimposed codes \cite{cms}.

\subsection{K\"{a}hler Groups}  The work of Delzant and Gromov \cite{DG}
highlighted the important role that subdirect products of surface groups play in the struggle to understand which 
finitely presented groups are fundamental groups of compact K\"{a}hler manifolds. In a subsequent article,
Claudio Llosa Isenrich and I will exploit the constraints that he established in \cite{claudio2} to investigate whether  
binary subgroups of direct products of surface groups can be K\"{a}hler.

\subsection{Implementation} Picking up on the theme of Corollary \ref{c:H_1enough} and Remarks \ref{r:algo} and \ref{r:explore},
I would be interested to know how practical it is to implement the algorithms whose existence is discussed in this article, in particular
the algorithm that takes as input  a list $\sigma$ of $m$ distinct positive integers and gives as output (1) a finite presentation of 
the binary subgroup $B<F_2^m$ that it defines, and (2) the maximum $k$ such that $B$ is of type $\wFP_k$. 

For (1), it would be necessary to implement a special case of the Effective 1-2-3 Theorem from \cite{BHMS2} --
cf.~\cite{claudio1}.


\begin{thebibliography}{99}
  
\bibitem{BMR}
G.~Baumslag, A.~Myasnikov and V.N.~Remeslennikov, 
{\em Algebraic geometry over groups
I: Algebraic sets and ideal theory,}
 J. Algebra, {\bf 219} (1999), 16--79.
 
\bibitem{bieri} R.~Bieri, {\em Homological dimension of discrete groups}, Queen Mary College Mathematical Notes. Queen Mary College, Department of Pure Mathematics, London, 1981. 

\bibitem{B:concise}
M.R.~Bridson, 
{\em Concise presentations of direct products},  
Proc. Amer. Math. Soc. {\bf{150}} (2022), 1361--1368. 


\bibitem{BW}
M.R.~Bridson and H.~Wilton,
{\em On the difficulty of presenting finitely presentable groups},
 Groups Geom. Dyn. {\bf 5} (2011), 301--325. 
 
 \bibitem{BHMS1} M.R.~Bridson, J.~Howie, C.F.~Miller III and H.~Short, 
 {\em Subgroups of direct products of limit groups},  Ann. of Math.  170  (2009),   1447--1467
  
\bibitem{BHMS2}  M.R.~Bridson, J.~Howie, C.F.~Miller III and H.~Short, 
 {\em On the finite presentation of subdirect products and 
the nature of residually free groups},
Amer. J. Math.,  {\bf{135}} (2013), 891--933.  


\bibitem{BL}
M.R.~Bridson and C.~Llosa Isenrich,
{\em K\"{a}hler groups, the co-nilpotency of subdirect products, and Sidki doubles}, 
 in preparation.

\bibitem{BM}
M.R.~Bridson  and C.F.~Miller,
\newblock {\em Structure and finiteness properties of subdirect products of groups}, 
\newblock Proc. Lond. Math. Soc. (3) {\bf 98} (2009),  631--651. 
 
 \bibitem{cms}  A.E.F.~Clementi, A.~Monti and R.~Silvestri, 
 {\em Selective families, superimposed codes, and broadcasting on unknown radio networks}, 
 12th ACM-SIAM SODA, 2001, pp. 709--718.
 
 
 \bibitem{DG} T.~Delzant and M.~Gromov,
Cuts in K\"ahler groups.
In {\em{Infinite Groups:Geometric, Combinatorial and Dynamical Aspects}},
Progress in Math. {\bf{248}}, Birkh\"auser, Basel 2005.
 
\bibitem{marshall-hall} M. Hall, Jr., ``The theory of groups", MacMillan Co., New York,  1959.

\bibitem{hall} P. Hall,
{\em The Eulerian functions of a group}, Quart. J. Math {\bf{7}} (1936), 134--151. 
  

\bibitem{Koch} D.H.~Kochloukova,  {\em On subdirect products of type $\FP_m$ of limit groups}, J. Group Theory {\bf{13}} (2010),   1--19.

\bibitem{Kuckuck} B.~Kuckuck,
{\em Subdirect products of groups and the $n$-$(n+1)$-$(n+2)$ conjecture}, 
Quart. J. Math. {\bf{64}} (2015), 1293--1318.

\bibitem{claudio1} C.~Llosa Isenrich,
{\em Finite presentations for K\"{a}hler groups with arbitrary finiteness properties},
J.~Algebra {\bf{476 }}(2017), 344--367.

\bibitem{claudio2} C.~Llosa Isenrich,
{\em Complex hypersurfaces in a direct product of Riemann surfaces}, arXiv:1806.02357. 


\bibitem{miller} C.~F.~Miller~III, 
``On group-theoretic decision problems and their classification",
Ann. Math. Studies,
No. 68, Princeton University Press (1971).

\bibitem{ashot} A.~Minasyan,
{\em On subgroups of right angled Artin groups with few generators}, 
Intl.~J.~Alg.~Comp. {\bf 25} (2015), 675--688.

\bibitem{sela} Z.~Sela,  {\em Diophantine geometry over groups I:
Makanin-Razborov diagrams},
IH{\'{E}}S Publ. Math. {\bf{93}} (2001), 31--105. 


\bibitem{stall}
J. R. Stallings, {\em A finitely presented group whose 3–dimensional homology group
is not finitely generated}, Amer. J. Math. , {\bf{85}} (1963) 541--543.
 
\bibitem{wiegold} J. Wiegold, {\em Growth sequences of finite groups IV},
 J. Austral. Math. Soc. Ser. {\bf{29}} (1980), 14--16.
 
\bibitem{WW} J. Wiegold and J.S. Wilson, 
{\em{Growth sequences of finitely generated groups}},
Arch. Math. {\bf{30}}
(1978), 337--343. 

 \end{thebibliography}
\end{document}